\documentclass[11pt, english]{amsart}

\usepackage{amsmath,amssymb,enumerate}

\usepackage[T1]{fontenc}
\usepackage[all,cmtip]{xy}

\usepackage{babel}
\usepackage{amstext}
\usepackage{amsmath}
\usepackage{amsfonts}
\usepackage{latexsym}
\usepackage{ifthen}

\usepackage{xypic}
\xyoption{all}
\usepackage[top=2.5cm,bottom=3cm,left=3.2cm,right=3.2cm]{geometry}

\newcommand{\rk}{{\rm rk}}
\newcommand{\codim}{{\rm codim}}

\newtheorem{lemma1}{}[section]

\newenvironment{lemma}{\begin{lemma1}{\bf Lemma.}}{\end{lemma1}}
\newenvironment{example}{\begin{lemma1}{\bf Example.}\rm}{\end{lemma1}}

\newenvironment{theorem}{\begin{lemma1}{\bf Theorem.}}{\end{lemma1}}

\newenvironment{proposition}{\begin{lemma1}{\bf Proposition.}}{\end{lemma1}}
\newenvironment{corollary}{\begin{lemma1}{\bf Corollary.}}{\end{lemma1}}
\newenvironment{remark}{\begin{lemma1}{\bf Remark.}\rm}{\end{lemma1}}

\newenvironment{definition}{\begin{lemma1}{\bf Definition.}}{\end{lemma1}}

\newenvironment{setup}{\begin{lemma1}{\bf Setup.}}{\end{lemma1}}

\newenvironment{problem}{\begin{lemma1}{\bf Problem.}}{\end{lemma1}}
\newenvironment{assumption}{\begin{lemma1}{\bf Assumption.}}{\end{lemma1}}
\newenvironment{fact}{\begin{lemma1}{\bf Fact.}}{\end{lemma1}}

\newenvironment{remark*}{{\bf Remark.}}{}
\newenvironment{remarks*}{{\bf Remarks.}}{}
\newenvironment{example*}{{\bf Example.}}{}
\newenvironment{assumption*}{{\bf Assumption.}}{}

\newcommand{\R}{\ensuremath{\mathbb{R}}}
\newcommand{\Q}{\ensuremath{\mathbb{Q}}}
\newcommand{\Z}{\ensuremath{\mathbb{Z}}}
\newcommand{\C}{\ensuremath{\mathbb{C}}}
\newcommand{\N}{\ensuremath{\mathbb{N}}}
\newcommand{\PP}{\ensuremath{\mathbb{P}}}

\usepackage{eurosym}

\newcommand{\holom}[3]{\ensuremath{#1:#2  \rightarrow #3}}
\newcommand{\fibre}[2]{\ensuremath{#1^{-1} (#2)}}

\makeatletter
\ifnum\@ptsize=0  \fi \ifnum\@ptsize=2  \fi 

\newcommand\sI{{\mathcal I}}

\newcommand\sN{{\mathcal N}}
\newcommand\sO{{\mathcal O}}

\DeclareMathOperator*{\pic}{Pic}
\DeclareMathOperator*{\sing}{sing}

\newcommand{\NE}[1]{ \ensuremath{ \overline { \mbox{NE} }(#1)} }

\usepackage{xcolor}
\usepackage{hyperref}

\newcommand{\AX}{\ensuremath{|-K_X|}}
\newcommand{\BsAX}{\ensuremath{\mbox{\rm Bs}}(|-K_X|)}

\newcommand{\Bs}[1]{\ensuremath{\mbox{\rm Bs}}(#1)}

\author{Andreas H\"oring}
\author{Saverio Andrea Secci}

\address{Andreas H\"oring, Universit\'e C\^ote d'Azur, CNRS, LJAD, France, Institut Universitaire de France}
\email{andreas.hoering@univ-cotedazur.fr}

\address{Saverio A.\ Secci,  Universit\`a di Torino, Dipartimento di Matematica, via Carlo Alberto 10, 10123 Torino - Italy}
\curraddr{Universit\`a di Milano, Dipartimento di Matematica, via Cesare Saldini 50, 20133 Milano - Italy}
\email{saverio.secci@unimi.it}


\title{Fano fourfolds with large anticanonical base locus} 
\date{\today} 

\begin{document}

\begin{abstract} 
A famous theorem of Shokurov states that a general anticanonical divisor of a smooth Fano threefold is a smooth K3 surface. This is quite surprising since there are several examples where the base locus of the anticanonical system has codimension two. In this paper we show that for four-dimensional Fano manifolds the behaviour is completely opposite: if the base locus is a normal surface, hence has codimension two, all the anticanonical divisors are singular. 
\end{abstract} 

\maketitle

\section{Introduction}

\subsection{Motivation and main result}

The anticanonical system is arguably the most natural object attached to a Fano manifold, 
and it plays an important role in the classification of Fano manifolds of low dimension. While it is expected that the anticanonical bundle always has global sections \cite{Cet90, Amb99,Kaw00}, it is in general not globally generated. In dimension three Shokurov's theorem gives a complete description of the situation:

\begin{theorem} \cite{Sho80} \label{theorem-shokurov}
Let $X$ be a smooth Fano threefold such that the base locus $\BsAX$ is not empty. 
Then the base locus is a smooth rational curve.
Moreover a general anticanonical divisor $Y \in \AX$ is smooth.
\end{theorem}

In dimension four our current knowledge about the anticanonical system is very limited.
An example by the first named author and Voisin shows that Shokurov's theorem does not generalise to higher dimension:

\begin{example} \cite[Ex.2.12]{HV11} \label{example-HV}
Let $S$ be a smooth del Pezzo surface of degree one, and
denote by $p \in S$ the unique base point of $|-K_S|$.
Set $X:=S \times S$, and let $Y \in \AX$ be a general anticanonical divisor. Then $Y$ is singular in the point $(p,p)$.
\end{example}

Indeed the threefold $Y$ contains the base locus
$\BsAX = p \times S \cup S \times p$ which has embedding dimension four in the point $(p,p)$.
The varieties in Example \ref{example-HV} belong to the unique family of Fano fourfolds having the maximal Picard number 18 \cite[Thm.1.1]{Cas23}, so one might hope that 
the presence of singular general anticanonical divisors
is a rare pathology. 
In this paper we destroy this hope by showing that 
a sufficiently large base locus always leads to singularities:

\begin{theorem} \label{theorem-main}
Let $X$ be a smooth Fano fourfold such that $h^0(X, \sO_X(-K_X)) \geq 3$.
Assume that the base locus $\BsAX$ is an irreducible normal surface. Then a general anticanonical divisor $Y \in \AX$
is not $\Q$-factorial, in particular it is singular.
\end{theorem}

This theorem covers all the examples we know of smooth Fano fourfolds such that $\BsAX$ is an irreducible surface, cf.\ Examples \ref{example1}, \ref{example2}, \ref{example2bis}. These examples differ from Example \ref{example-HV} in a significant way: they have moving singularities, i.e.\ the singular locus $Y_{\sing} \subset \BsAX$
depends on the choice of $Y \in \AX$.
Theorem \ref{theorem-main} is almost optimal: if the base scheme $\mbox{\bf Bs}(|-K_X|)$
is smooth of dimension at most one, a general anticanonical divisor $Y$ is smooth by a strong version of Bertini's theorem
\cite[Cor.2.4]{DH91}. It seems likely that Theorem \ref{theorem-main} still holds under the relaxed assumption that $\BsAX$ {\em contains} a surface. However this additional generality leads to numerous case distinctions in our proofs, a technicality that we wanted to avoid.

While our main result states that the anticanonical geometry of a Fano fourfold
is more complicated than in dimension three, the tools developed in this paper indicate that
two-dimensional base loci lead to numerous
restrictions on the geometry and the numerical invariants.
Therefore Theorem \ref{theorem-main}  is actually a first step towards the classification of Fano fourfolds with large anticanonical base locus, we plan to come back to this topic in the near future.

\subsection{The setup} \label{setup} 

Let $X$ be a smooth Fano fourfold such that $h^0(X, \sO_X(-K_X)) \geq 3$
and $\BsAX$ is an irreducible normal surface $B$.

Let $Y \in \AX$ be a general anticanonical divisor, so $Y$ is a normal threefold (Corollary \ref{theorem-known}). By Kodaira vanishing the restriction morphism
$$
H^0(X, \sO_X(-K_X)) \rightarrow H^0(Y, \sO_Y(-K_X))
$$
is surjective.  Thus we know that
$$
\Bs{|-K_X|_Y|} = \BsAX = B
$$
and $h^0(Y, \sO_Y(-K_X))=h^0(X, \sO_X(-K_X))-1 \geq 2$.
Denote by $|M|$ the mobile part of the linear system $|-K_X|_Y|$, then we have a linear equivalence of effective Weil divisors 
$$
-K_X|_Y \simeq M + B.
$$
The proof of Theorem \ref{theorem-main} will be by contradiction, so we make the

\begin{assumption} \label{main-assumption}
We assume that $Y$ is $\Q$-factorial.
\end{assumption}

Combined with the known results about anticanonical divisors, cf.\ Corollary \ref{corollary-factorial} this assumption implies that the Weil divisors $M$ and $B$ are Cartier.
By Grothendieck's Lefschetz hyperplane theorem \cite[Example 3.1.25]{Laz04a} the restriction morphism 
$$
\pic(X) \rightarrow \pic(Y)
$$
is an isomorphism, so there exists uniquely defined Cartier divisor classes $M_X \rightarrow X$
and $B_X \rightarrow X$ such that
$$
M \simeq M_X|_Y, \qquad B \simeq B_X|_Y.
$$
Note that it is not clear whether the divisors $M_X$ and $B_X$ are effective, establishing this will be an important first step in our proof.

\subsection{General strategy and structure of the proof}
The proof of Shokurov's theorem \ref{theorem-shokurov} can be split into two steps, cf.\ \cite[\S 2.3]{IP99}: first one shows that $Y$ has canonical singularities. Then one applies Mayer's theorem to a minimal resolution $Y' \rightarrow Y$.

\begin{theorem} \label{theorem-mayer} \cite[Prop.5]{May72}
Let $S$ be a smooth K3 surface, and let $A$ be a nef and big Cartier divisor on $S$ such that $\Bs{|A|} \neq \emptyset$. Then we have a decomposition into fixed and mobile part
$$
A \simeq M + B
$$
where $B \simeq \PP^1$ and $M$ is a nef divisor that defines an elliptic fibration $\varphi_{|M|}: S \rightarrow \PP^1$.
\end{theorem}

Since the first part of  Shokurov's proof also works for Fano fourfolds \cite[Thm.5.2]{Kaw00}, the most natural approach would be to look for an analogue of Mayer's theorem for Calabi-Yau threefolds (with mild singularities). 
This approach immediately encounters two obstacles:
while it is easy to classify fixed prime divisors on a K3 surface (they are $(-2)$-curves), there are many possibilities on a Calabi-Yau threefold, e.g.\ Enriques surfaces \cite[Thm.3.1]{BN16}.
Moreover the mobile part of a linear system on a surface is always nef, this is not true on a threefold.

The second obstacle leads to a basic case distinction in our proof: if $M_X$ is nef, 
the basepoint free theorem yields a morphism $X \rightarrow T$ that we can use to study the anticanonical system.
This situation is close to the Examples \ref{example1}, \ref{example2}, \ref{example2bis}, and we will need a series of rather specific classification results to settle this case 
in Section \ref{section-nef-case}.

In the second case, $M_X$ is not nef, we use the embedding $Y \subset X$ in a more direct way. We start by showing that $-K_X+M_X$ is nef and big, i.e.\ the anticanonical divisor 
compensates the negativity of $M_X$. This allows us to show
that the surface $B$ is a complete intersection:
we find an effective divisor $B_X$ with canonical singularities such that 
$$
B = Y \cap B_X.
$$
Moreover, using an extension theorem of Fujino, we obtain that
the restriction map
$$
H^0(X, \sO_X(-K_X)) \rightarrow H^0(B_X, \sO_{B_X}(-K_X))
$$
is surjective. Combined with some classification results 
for linear systems on irregular surfaces, cf.\ Section \ref{section-auxiliary-I}, we exclude this possibility.

\subsection{Future research}

In Theorem \ref{theorem-main} we make the assumption that 
$h^0(X, \sO_X(-K_X)) \geq 3$, an assumption that is satisfied by all smooth Fano fourfolds that we are aware of\footnote{There is an example with $h^0(X, \sO_X(-K_X)) =3$; see \cite[Prop.2.2]{Kue97}\cite[Table 2]{Qur21}.}. Riemann-Roch computations only show that $h^0(X, \sO_X(-K_X)) \geq 2$ (Theorem \ref{theorem-known}), so it would be highly desirable to have an answer to the following 

\begin{problem} \cite{Fuj88two,Kue97, Liu20}
Is there a smooth Fano fourfold $X$ with $h^0(X, \sO_X(-K_X)) = 2$?

Is there a smooth Calabi-Yau threefold  $Y$ with an ample Cartier divisor $A$ such that $h^0(Y, \sO_Y(A))=1$?
\end{problem}

Beauville \cite{Bea99b} gave an example of a numerical Calabi-Yau threefold with a fixed ample divisor, for our purpose we are interested in strict Calabi-Yau's, i.e.\ $Y$ is simply connected.

A significant problem in this theory is the lack of interesting examples. Our Example \ref{example2bis} generalises a construction from the threefold case, but it is still related to the del Pezzo surface of degree one. 

\begin{problem}
Construct new examples of smooth Fano fourfolds $X$ such that $\BsAX$ has dimension two.
\end{problem}

{\bf Acknowledgements.} A.H.\ thanks the Institut Universitaire de France for providing excellent working conditions for this project. S.A.S.\ thanks the LJAD Universit\'e C\^ote d'Azur for partially funding his long stay in Nice, in order to establish this project; he also thanks the Mathematics Department of the University of Turin for the opportunities given during his Ph.D. We thank the referee for the careful verification of our arguments.

\section{Notation and basic facts}

We work over $\C$, for general definitions we refer to \cite{Har77}.
All the schemes appearing in this paper are projective, manifolds and normal varieties will always be supposed to be irreducible.
For notions of positivity of divisors and vector bundles we refer to Lazarsfeld's book \cite{Laz04a, Laz04b}. 
Given two Cartier divisors $D_1, D_2$ on a projective variety we denote by $D_1 \simeq D_2$ (resp. $D_1 \equiv D_2$) the linear equivalence (resp. numerical equivalence) of the
Cartier divisor classes.
 Given a Cartier divisor $D$ we will denote by $\sO_X(D)$ both the associated invertible sheaf and the corresponding line bundle.
Somewhat abusively we will say that a Cartier divisor class is effective if it contains an effective divisor.

We use the terminology of \cite{KM98} for birational geometry and of \cite{Kol96} for rational curves.
We refer to \cite{Kol13} for the definitions and basic facts about singularities of the MMP.

Given a normal projective surface $S$ with Gorenstein singularities we will denote by
$$
q(S) := h^1(S, \sO_S)
$$
the irregularity, and by
$$
p_g(S) := h^2(S, \sO_S) = h^0(S, \sO_S(K_S))
$$
the geometric genus.

\begin{definition}
Let $Y$ be a projective variety, and let $D$ be a nef Cartier divisor on $Y$. The numerical dimension $\nu(Y,D)$ is defined as 
$$
\max\{k \in \N \ | \ D^k \neq 0 \}.
$$
\end{definition}

We collect a number of basic facts for the convenience of the reader:

\begin{fact} \label{fact-not-moving}
Let $A$ be a Cartier divisor on a projective variety $Y$ such that the complete linear system $|A|$ has a non-empty fixed part $B$. Then $h^0(Y,\sO_Y(B))=1$. 
\end{fact}

\begin{fact} \label{fact-bpf-application}
Let $X$ be a Fano manifold, and let $D$ be a nef Cartier divisor on $X$. Then there exists a
morphism with connected fibres $\holom{\varphi}{X}{T}$ and an ample Cartier divisor $H_T$ on $T$ such that $D \simeq \varphi^* H_T$.
\end{fact}

\begin{proof}
By the basepoint free theorem we have $mD \simeq \varphi^* H_m$ for some very ample divisor $H_m \rightarrow T$ for every $m \gg 0$. Thus $H_T := H_{m+1}-H_m$ is a Cartier divisor such that $D \simeq \varphi^* H_T$. Since $D$ is numerically the pull-back of an ample class on $T$, the divisor $H_T$ is ample.
\end{proof}

\begin{fact} \label{fact-deformations-surface}
Let $S$ be an irreducible projective surface with Gorenstein singularities, and let $C \subset S$
be an irreducible curve that is not contained in the singular locus of $S$. If $-K_S \cdot C \geq 2$, the curve $C$ deforms in $S$ in a positive-dimensional family.
\end{fact}

\begin{proof}
Let $\holom{\tau}{S'}{S}$ be the composition of normalisation and minimal resolution.
Then we have $K_{S'} \simeq \tau^* K_S - E$ where $E$ is an effective divisor that
maps into the singular locus of $S$. Thus if $C' \subset S'$ is the strict transform, it is not contained in the support of $E$. Thus we have
$$
-K_{S'} \cdot C' \geq -K_S \cdot C \geq 2.
$$
The statement now follows from \cite[Thm.1.15]{Kol96}.
\end{proof}

\begin{fact} \label{fact-albanese} \cite[Lemma 2.4.1]{BS95}
Let $X$ be a normal projective variety with rational singularities such that $q(X)>0$.
Then the Albanese map $\holom{\alpha}{X}{A}$ to the Albanese torus exists and is determined
by the Albanese map of some resolution of singularities.
\end{fact}

\begin{fact} \label{fact-RR}
Let $S$ be a normal projective surface with rational singularities, and let $A$ be a Cartier divisor on $S$. Then the Riemann-Roch formula 
$$
\chi(S, \sO_S(A)) = \frac{1}{2} A^2 + \frac{1}{2} (-K_S) \cdot A + \chi(S, \sO_S)
$$
holds.

Let $Y$ be a normal projective threefold with terminal singularities,
and let $A$ be a Cartier divisor on $Y$. Then the Riemann-Roch formula 
$$
\chi(Y, \sO_Y(A)) = 
\frac{1}{12} A \cdot (A-K_Y) \cdot (2A-K_Y) + \frac{1}{12} A \cdot c_2(Y)
+ \chi(Y, \sO_Y)
$$
holds.
\end{fact}

\begin{proof}
It is sufficient to verify that all the terms are well-defined and can be calculated on a resolution of singularities. This is well-known, cf.\ \cite[Sect.2]{Hoe12}, the statement then follows via the projection formula.
\end{proof}

In the rest of the paper we will refer to the following set of statements by 
inversion of adjunction:

\begin{theorem} \label{theorem-adjunction} \cite[Thm.5.50]{KM98} \cite[Thm.]{Kaw07} \cite[Thm.4.9]{Kol13}
Let $X$ be a normal projective variety, and let $S \subset X$ be a reduced Weil divisor that is Cartier in codimension two. Let $\Delta$ be an effective boundary divisor on $X$ that has no common component with $S$ such that $K_X+S+\Delta$ is $\Q$-Cartier. Then the following holds:
\begin{itemize}
\item The pair $(X, S+\Delta)$ is lc near $S$
if and only if the pair $(S, \Delta|_S)$ is slc.
\item The pair $(X, S+\Delta)$ is plt near $S$
if and only if the pair $(S, \Delta|_S)$ is klt.
\end{itemize}
\end{theorem}

\subsection{Anticanonical divisors on Fano fourfolds}

The following statement collects the known results about anticanonical divisors
on smooth Fano fourfolds 

\begin{theorem} \label{theorem-known} \cite[Thm.5.2]{Kaw00},\cite[Thm.1.7]{HV11}, \cite[Thm.2]{Heu15}
Let $X$ be a smooth Fano fourfold. Then we have
$$
h^0(X, \sO_X(-K_X)) \geq 2.
$$
Let $Y \in \AX$ be a general anticanonical divisor. Then $Y$ is a normal prime divisor
with terminal Gorenstein singularities. The variety $Y$ is Calabi-Yau, i.e.\ $K_Y \simeq \sO_Y$ and
$$
H^1(Y, \sO_Y) = H^2(Y, \sO_Y) = 0.
$$
\end{theorem}

In fact the possible singularities of $Y$ are completely described by \cite[Thm.4]{Heu15}, but we will not need this description for our proof. 
By \cite[Lemma 5.1]{Kaw88} a terminal $\Q$-factorial Gorenstein threefold singularity is factorial, so Theorem \ref{theorem-known} implies:

\begin{corollary} \label{corollary-factorial}
Let $X$ be a smooth Fano fourfold. If a general anticanonical divisor $Y$ is $\Q$-factorial,
it is factorial.
\end{corollary}

\begin{theorem} \cite[Theorem]{Bor91} \label{theorem-kollar}
Let $X$ be a Fano manifold of dimension at least four, and let $Y \in |-K_X|$ be an irreducible\footnote{The statement in \cite{Bor91} is for smooth divisors, but the proof works without this assumption.} anticanonical divisor. Then the inclusion
$$
\NE{Y} \hookrightarrow \NE{X}
$$
is a bijection. Equivalently a Cartier divisor $D \rightarrow X$ is nef if and only if
its restriction $D|_Y$ is nef.
\end{theorem}

\begin{theorem} \label{theorem-kawamata} \cite[Prop.4.2]{Kaw00}
Let $Y$ be a normal projective threefold with canonical Gorenstein singularities such that $c_1(Y)=0$.  Let $A \rightarrow Y$ be an ample Cartier divisor, and let $D \in |A|$ be a general element. Then the pair $(Y, D)$ is lc.
\end{theorem}

\begin{corollary} \label{corollary-restrictions-singularities}
In the situation of Setup \ref{setup}, let $M$ be a general element of the mobile part
of $|-K_X|_Y|$. Then the pair $(Y, M+B)$ is lc, hence the pairs $(Y, M)$ and $(Y,B)$
are lc. In particular $B$ is a normal projective Gorenstein surface with at most lc singularities and
$p_g(B)=0$. 
Moreover
$$
(B, M|_B)
$$
is an lc pair such that 
\begin{equation}
\label{A-is-adjoint}
-K_X|_B \simeq K_B + M|_B  
\end{equation}
is ample with $(K_B+M|_B) \cdot M|_B>0$. 
\end{corollary}

\begin{proof}
Since $M$ is a general divisor in the mobile part, the divisor $M+B$ is general in $|-K_X|_Y|$.
Thus $(Y, M+B)$ is lc by Theorem \ref{theorem-kawamata}. The statement for $(Y, M)$ and $(Y,B)$
is now immediate since we assume that $Y$ is $\Q$-factorial. By adjunction (Theorem \ref{theorem-adjunction}) this implies that the surfaces $M$ and $B$ have at most slc singularities. Since $B$ is normal by assumption, it has lc singularities. Consider the cohomology sequence associated to the exact sequence
$$
0 \rightarrow \sO_Y \rightarrow \sO_Y(B) \rightarrow \sO_B(B) \simeq \sO_B(K_B) \rightarrow 0.
$$
The Calabi-Yau threefold $Y$ has $h^1(Y, \sO_Y)=0$ and $h^0(Y, \sO_Y(B))=1$ by Fact \ref{fact-not-moving}, so $p_g(B)=h^0(B, \sO_B(K_B))=0$.

Again by adjunction the pair $(B, M|_B)$ is slc (hence lc) since $(Y, M+B)$ is lc.
Since $K_Y$ is trivial, the linear equivalence \eqref{A-is-adjoint} follows from the adjunction formula. Finally observe that the support of the divisor $M+B \simeq -K_X|_Y$ is connected since it is ample. Therefore the intersection $M \cap B$ has positive dimension and 
$(K_B+M|_B) \cdot M|_B = -K_X|_B \cdot M|_B>0$. 
\end{proof}

\section{Examples and first observations}

We start this section by collecting some arguments that give a moral explanation 
for Theorem \ref{theorem-main}.

\begin{lemma} \label{lemma-example}
Let $X$ be a projective manifold, and let $A$ be an ample divisor on $X$. Let  $B \subset \Bs{|A|}$ be a submanifold of dimension at least
$\frac{\dim X}{2}$ such that the conormal bundle $\sN_{B/X}^*$ is nef. Then every divisor $Y \in |A|$ has at least one singular point along $B$. 
\end{lemma}

\begin{proof}
Consider the natural morphism
$$
\alpha: H^0(X, \sO_X(A)) \simeq H^0(X, \sI_B \otimes \sO_X(A)) \rightarrow H^0(B, \sI_B/\sI_B^2 \otimes \sO_X(A)).
$$
By \cite[Lemma 1.7.4]{BS95} the divisor $Y$ is smooth along $B$ if and only if the section 
$\alpha(s) \in H^0(B, \sI_B/\sI_B^2 \otimes \sO_X(A))$ is nowhere zero, where $0 \neq s \in H^0(X, \sO_X(A))$ is a section vanishing on $Y$. 
Since $\sI_B/\sI_B^2 \otimes \sO_X(A) \simeq \sN_{B/X}^* \otimes \sO_X(A)$ is a tensor product of a nef vector bundle and an ample line bundle, it is ample \cite[Thm.6.2.12.(iv)]{Laz04b}.
Set $d=\codim_X B$. If $\alpha(s)$ does not vanish in any point of $B$, the top Chern class
$c_d(\sN_{B/X}^* \otimes \sO_X(A))$ is zero. Yet since $d \leq \dim B$ this is a contradiction
to the positivity of Chern classes of ample vector bundles \cite[Thm.8.3.9]{Laz04b}.
\end{proof}

\begin{remark*}
The key point of the proof above is that the vector bundle $\sN_{B/X}^* \otimes \sO_X(A)$ is ample and one may wonder if this holds without any assumption on $\sN_{B/X}^*$: as a toy case, assume that the {\em base scheme} $\mbox{\bf Bs}(|A|)$ is the submanifold $B$. We can resolve the indeterminacies of $\varphi_{|A|}$
by blowing up $B$ and compute that $\sN_{B/X}^* \otimes \sO_X(A)$ is nef. However it is not clear if the vector bundle is ample, in fact this is impossible if $h^0(X, \sO_X(A)) \leq \dim X-1$.
\end{remark*}

The lemma allows to cover all the known examples of Fano fourfolds with a two-dimensional base locus:

\begin{example} \label{example1}
Let $X = S_1 \times S$ be the product of smooth del Pezzo surface $S_1$ of degree one 
and $S$ a smooth del Pezzo surface of degree at least two. Then every anticanonical divisor $Y \in \AX$
is singular. 

{\em Proof.} The base locus is $B=p \times S$ where $p =\mbox{Bs}(|-K_{S_1}|)$.
Thus $\sN^*_{B/X} \simeq \sO_B^{\oplus 2}$ is trivial and Lemma \ref{lemma-example} applies.
\end{example} 

Applied to an irreducible component of the base locus the argument above also works for \cite[Ex.2.12]{HV11}.

\begin{example} \label{example2}
Let $X = Z \times \PP^1$ where $Z$ is a smooth Fano threefold
such that $C :=\Bs{|-K_{Z}|}$ is not empty. Then every anticanonical divisor $Y \in \AX$
is singular. 

{\em Proof.} The base locus is $B=C \times \PP^1$ and we know that 
$\sN^*_{C/Z} \simeq \sO_C^{\oplus 2}$ or $\sN^*_{C/Z} \simeq \sO_C \oplus \sO_C(1)$ (cf.\ the proof of \cite[Lemma 2.4.4]{IP99}). Thus $\sN^*_{B/X}$ is nef 
and Lemma \ref{lemma-example} applies.
\end{example} 

The next manifold is a generalisation of a well-known example in the threefold case \cite[Table 2, No 1]{MM81}.

\begin{example} \label{example2bis}
For $n \geq 4$ let $M$  be a general sextic hypersurface in the weighted projective space $\PP(1^n, 2, 3)$. Then $M$ is a del Pezzo manifold of dimension $n$ and degree $1$ \cite[Thm.8.11]{Fuj90}, i.e.\ we have $-K_M \simeq (n-1) A$ with $A$ an ample Cartier divisor such that $A^n=1$.
The del Pezzo manifold $M$ has a ladder \cite[Thm.3.5]{Fuj90}, i.e.\ we can choose $n-1$ general elements $D_1, \ldots, D_{n-1}$ in the linear system $|A|$ such that 
$$
C:= D_1 \cap \ldots \cap D_{n-1}
$$
is a smooth elliptic curve. Morever $|A|$ has a unique base-point $p$. 

Let now
$\holom{\mu}{X}{M}$ be the blow-up along $C$, and denote by $E$ the exceptional divisor. Then $\mu$ resolves the base locus of the linear system $|V|$ generated by $D_1, \ldots, D_{n-1}$ so we have a fibration
$$
\holom{f}{X}{\PP^{n-2}}
$$
such that $\mu^* A \simeq f^* H + E$ where $H$ is the hyperplane divisor on $\PP^{n-2}$. Using the Nakai-Moishezon criterion we obtain immediately that
$$
-K_X \simeq \mu^* (-K_M) - (n-2) E  \simeq (n-1) \mu^* A - (n-2) E \simeq (n-2) f^* H + \mu^* A
$$
is ample, so $X$ is a Fano manifold. The general $f$-fibre is a surface $S$ obtained as the complete intersection of $n-2$ elements in $|V|$, so it is a del Pezzo surface of degree one.
In particular $|-K_S|$ has a unique base point $p_S$ such that $\mu(p_S)=p$. This implies that $\fibre{\mu}{p} \simeq \PP^{n-2}$ is contained in $\BsAX$. 

Since $\sN^*_{\PP^{n-2}/X}
\simeq \sO_{\PP^{n-2}} \oplus \sO_{\PP^{n-2}} (1)$ is nef we can apply Lemma \ref{lemma-example}
to see that all the anticanonical divisors in $X$ are singular.
\end{example}

\begin{remark} \label{remark-factorial-singular}
Let $X$ be a complex manifold, and let $B \subset X$ be a submanifold of codimension two.
Let $B \subset Y \subset X$ be a normal prime divisor. Then $Y$ is singular in a point $p \in B$
if and only if $Y$ is not factorial in $p$. Indeed if $Y$ is factorial in $p$, the Weil divisor $B \subset Y$ is Cartier near the point $p$. Thus if $Y$ is singular in $p$, so is the Cartier divisor $B$.
\end{remark}

The next statement shows that for most prime Fano manifolds (i.e.\ the prime Fanos with index one) the presence of a codimension two base locus is incompatible with the smoothness of the anticanonical divisor.

\begin{proposition} \label{proposition-picard-number-one}
Let $X$ be a Fano manifold of dimension at least four with Picard number one such that $h^0(X, \sO_X(-K_X)) \geq 3$.
Assume that the base locus $\BsAX$ has codimension two and a general anticanonical divisor
$Y \in \AX$ is smooth. Then $X$ has Fano index at least three.
\end{proposition}

\begin{proof} 
Let 
$$
-K_X|_Y \simeq M + B
$$
be the decomposition into fixed and mobile part. Since $Y$ is smooth, the Weil divisors 
$M$ and $B$ are Cartier. Since $\dim X \geq 4$ and $\rho(X)=1$, the Lefschetz hyperplane theorem implies $\pic{Y} \simeq \pic{X} \simeq \Z H$ where $H$ is the ample generator of the Picard group.
Thus we have
$$
M \simeq a H|_Y, \qquad B \simeq b H|_Y
$$
with $a,b \in \N^*$. Since $h^0(Y, \sO_Y(B))=1$ and $h^0(Y, \sO_Y(M)) \geq 2$ we have $a>b$ and therefore $a+b \geq 3$.
Since
$$
-K_X|_Y \simeq M+B \simeq (a+b) H|_Y, 
$$
the statement follows from $\pic{Y} \simeq \pic{X}$.
\end{proof}

\begin{lemma} \label{lemma-two-sections-fano}
Let $X$ be a smooth Fano fourfold, and let $D$ be a non-zero nef Cartier divisor on $X$.
Then $h^0(X, \sO_X(D)) \geq 2$.
\end{lemma}

\begin{proof}
 If $D$ has numerical dimension one, there exists a fibration $\varphi: X \rightarrow \PP^1$
 such that $D \simeq \varphi^* H$ with $H$ an ample Cartier divisor on $\PP^1$ (cf.\ Fact \ref{fact-bpf-application}). Thus $h^0(\PP^1, \sO_{\PP^1}(H)) \geq 2$ yields the claim.

If $D$ has numerical dimension at least two, we have $D^2 \cdot K_X^2= D^2 \cdot (-K_X)^2>0$. Moreover $h^0(X, \sO_X(D))=\chi(X, \sO_X(D))$ by Kodaira vanishing and by Riemann-Roch
\begin{eqnarray*} 
\chi(X, \sO_X(D)) &=& \frac{D^4}{4!} - \frac{D^3 \cdot K_X}{2 \cdot 3!} + \frac{D^2 \cdot (K_X^2+c_2(X))}{12 \cdot 2 !}
- \frac{D \cdot K_X \cdot c_2(X)}{24 \cdot 1!} + \chi(X, \sO_X)
\\
& \geq &
\frac{D^2 \cdot (-K_X)^2}{24}
+
\frac{D^2 \cdot c_2(X)}{24}
- \frac{D \cdot K_X \cdot c_2(X)}{24} + 1.
\end{eqnarray*}
Since $-K_X \cdot D \cdot c_2(X) \geq 0$ and $D^2 \cdot c_2(X) \geq 0$ by \cite[Thm.1.3]{Pet12}, \cite[Thm.6.1]{Miy87}  we obtain that the second line is at least two.
\end{proof}

\begin{lemma} \label{lemma-pair-lc}
Let $X$ be a smooth Fano fourfold, and let $L$ be a line bundle on $X$ such that $h^0(X, L) \neq 0$. Let $D \in |L|$ be a general divisor, and assume that for a general anticanonical divisor $Y \in \AX$, the pair $(Y, D \cap Y)$ is log-canonical. Then the pair $(X,D)$ 
is log-canonical.
\end{lemma}

\begin{proof}
Since $\dim X \leq 4$  the anticanonical system $\AX$ has no fixed component, cf.\ Theorem \ref{theorem-known}. 
In particular $Y$ has no common component with $D$.

We argue by contradiction and assume that the non-lc locus $Z$ is not empty. 

{\em 1st step. Assume that $Z$ has positive dimension.}
Since $Y$ is an ample divisor, it intersects $Z$ in at least one point,
so the pair $(X,D)$ is not lc near $Y$. Thus the pair $(X, D+Y)$ is not lc near $Y$. By Theorem \ref{theorem-adjunction} this is a contradiction
to Corollary \ref{corollary-restrictions-singularities}.

{\em 2nd step. Assume that $Z$ has dimension zero.}
We endow $Z \subset X$ with its natural scheme structure, cf.\ \cite[\S 7]{Fuj11}.
Since
$$
L - (K_X+D) \simeq -K_X
$$
is ample we can apply Fujino's extension theorem (with $W$ the empty set, cf.\ \cite[Thm.8.1]{Fuj11}) to see that the restriction map
$$
H^0(X, L) \rightarrow H^0(Z, L)
$$
is surjective. Since $Z$ is a finite scheme, the restriction $L|_Z$ is globally generated.
Thus $L$ is globally generated near $Z$, in particular a general $D \in |L|$ is smooth near $Z$. Hence the pair $(X,D)$ is lc, a contradiction.
\end{proof}

\section{Auxiliary results about linear systems, part I}
\label{section-auxiliary-I}

In Section \ref{section-positivity} our study of the Fano fourfold $X$ will lead to several
restrictions on the geometry of the base locus $\BsAX$. These restrictions can be strengthened when combined with the technical results of this section.

\begin{lemma} \label{lemma-pg-zero}
Let $S$ be a normal projective surface with canonical singularities such that
$p_g(S)=0$. Assume that the Albanese map induces a fibration 
$\holom{\alpha}{S}{C}$ onto a curve of genus $g=q(S)>0$. Then the general $\alpha$-fibre
is $\PP^1$ or $g=1$. 
\end{lemma}

\begin{proof}
Since canonical surface singularities are rational, we can replace $S$ with a resolution of singularities. Thus $S$ is a smooth surface with $p_g(S)=0$ and $q(S) \geq 1$. 
If the general $\alpha$-fibre is not $\PP^1$, the surface $S$ is not uniruled.
Thus we can apply \cite[Thm.VI.13]{Bea83book} to the minimal model of $S$ to obtain that $q(S)=1$.
\end{proof}

\begin{proposition} \label{proposition-exceptional-cases-B}
Let $S$ be a normal projective Gorenstein surface with $p_g(S)=0$, and let $A$ be a nef and big Cartier divisor on $S$ such that
$$
A \simeq K_S + \Delta_S
$$
where $\Delta_S$ is an effective Weil divisor such that the pair $(S, \Delta_S)$ is lc
and $A \cdot \Delta_S>0$. Assume that we have
$$
h^0(S, \sO_S(A)) \leq q(S).
$$
Then $S$ has canonical singularities and the Albanese map induces a fibration 
$$
\holom{\alpha}{S}{C}
$$
onto a curve
of genus $q(S)>0$. Set $r:=\rk (\alpha_* \sO_S(A))$, then 
we have 
$$
(r-1) (q(S)-1) \leq 1.
$$ 
Moreover we have 
$$
h^0(S, \sO_S(A)) = q(S).
$$
\end{proposition}

\begin{remark*}
In the second step of the proof we will use the following fact that follows easily from the Leray spectral sequence: let $S$ be a normal projective surface with Gorenstein singularities
such that $p_g(S)=0$. Assume that $S$ has irrational (hence non-canonical \cite[Cor.5.24]{KM98}) singularities, and let $\holom{\tau}{S_c}{S}$ be the canonical modification \cite[Thm.1.31]{Kol13}. Then we have
$$
p_g(S_c)=p_g(S)=0, \qquad q(S_c)>q(S).
$$
\end{remark*}

\begin{proof}[Proof of Proposition \ref{proposition-exceptional-cases-B}]
Since $A \cdot \Delta_S>0$, we have $\Delta_S \neq 0$.
Thus 
$$
H^2(S, \sO_S(K_S+\Delta_S)) \simeq H^0(S, \sO_S(-\Delta_S))=0
$$
and we have the inequalities
$$
q(S) \geq h^0(S, \sO_S(K_S+\Delta_S)) \geq \chi(S, \sO_S(K_S+\Delta_S)).
$$

{\em 1st step. Assume that $S$ has canonical singularities.} By Fact \ref{fact-RR} we have the Riemann-Roch formula
$$
\chi(S, \sO_S(K_S+\Delta_S)) = \frac{1}{2} (K_S+\Delta_S) \cdot \Delta_S + \chi(\sO_S).
$$
Hence
$A \cdot \Delta_S>0$ implies that $q(S) > \chi(S, \sO_S) = 1 - q(S)$. 
In particular we have $q(S)>0$, so there is a non-trivial Albanese morphism $\alpha$. If $\dim \alpha(S)=2$, the ramification formula shows that $p_g(S)>0$ which we excluded. Thus the Albanese morphism gives a fibration \cite[Lemma 2.4.5]{BS95}
$$
\holom{\alpha}{S}{C}
$$
onto a curve $C$ of genus $g:=q(S)>0$.
Since the pair $(S, \Delta_S)$ is lc, the direct image sheaf
$$
\alpha_* \sO_S(K_{S/C}+\Delta_S) \simeq \alpha_* \sO_S(A) \otimes \sO_C(-K_C)
$$
is a nef vector bundle \cite[Thm.1.1]{Fuj17} of rank $r$. Thus
$$
V:= \alpha_* \sO_S(A) \simeq \alpha_* \sO_S(K_{S/C}+\Delta_S) \otimes \sO_C(K_C)
$$
has $c_1(V) \geq r (2g-2)$. By the Riemann-Roch formula on curves we have
\begin{equation}
\label{RR-estimate}
h^0(S, \sO_S(A)) = h^0(C, V) \geq \chi(C, V) = c_1(V) + r \chi(C, \sO_C)
\geq r (g-1).
\end{equation}
Now observe that $r(g-1) > g$ unless $(r-1) (g-1) \leq 1$.

Finally let us show that we have $h^0(S, \sO_S(A)) = q(S)$: if $h^0(S, \sO_S(A)) \leq q(S)-1$
the Riemann-Roch estimate \eqref{RR-estimate} becomes $g-1 \geq r (g-1)$.

{\em Subcase a) Assume that $g>1$.} 
Then the unique possibility $r=1$. Since $A$ is nef and big, this implies that the general $\alpha$-fibre is not $\PP^1$. Yet this contradicts 
Lemma \ref{lemma-pg-zero}.

{\em Subcase b) Assume that $g=1$.} 
In this case the Riemann-Roch inequality becomes
$$ 
h^0(S, \sO_S(K_S+\Delta_S)) \geq \frac{1}{2} (K_S+\Delta_S) \cdot \Delta_S>0,
$$
so $h^0(S, \sO_S(K_S+\Delta_S)) \geq 1 = g$.

{\em 2nd step. We show that $S$ has canonical singularities.}
Since $S$ is Gorenstein, it has canonical singularities if and only if it has rational singularities. Arguing by contradiction we assume that $S$ has non-canonical singularities and denote by $\holom{\tau}{S_c}{S}$ the canonical modification.
Then we have $K_{S_c} \simeq \tau^* K_S - E$ with $E$ an effective Weil divisor.
Observe that 
$$
K_{S_c} + E + \tau^* \Delta_S \simeq \tau^* (K_S + \Delta_S),
$$
so the pair $(S_c, \Delta_{S_c}) := (S_c, E + \tau^* \Delta_S)$ is lc.
Moreover
$$
A_c := \tau^* A \simeq K_{S_c} + \Delta_{S_c}
$$
is a nef and big Cartier divisor
with $A_c \cdot \Delta_{S_c}=A \cdot \Delta_S>0$ and $h^0(S, \sO_S(A)) = h^0(S_c, \sO_{S_c}(A_c))$.
Finally by the remark before the proof one has 
$p_g(S_c) \leq p_g(S)=0$ and $q(S_c)>q(S)$. Therefore by Step 1 of the proof
$$
h^0(S, \sO_S(A)) = h^0(S_c, \sO_{S_c}(A_c)) = q(S_c) > q(S),
$$
a contradiction to our assumption.
\end{proof}

The conditions in Proposition \ref{proposition-exceptional-cases-B} are very restrictive,
however there is a classical example that will play a prominent role
in Section \ref{section-nef-case}:

\begin{example} \label{example-q-one}
Let $C$ be an elliptic curve, and for some point $p \in C$ let
$$
0 \rightarrow \sO_C \rightarrow V \rightarrow \sO_C(p) \rightarrow 0
$$
be an unsplit extension. Set $\holom{\alpha}{S:=\PP(V)}{C}$ and denote by $\zeta_S$ the tautological class on $S$. Set $\Delta_S := 3 \zeta_S - \alpha^* p$, then $\Delta_S$ is an ample divisor
\cite[V, Prop.2.21]{Har77}
such that $|\Delta_S|$ contains an element with normal crossing singularities.
Thus the pair $(S, \Delta_S)$ is lc and 
$$
A :=K_S+\Delta_S \simeq \zeta_S 
$$
is an ample divisor with $h^0(S, \sO_S(A))=1=q(S)$.
\end{example}

Since it would be tedious to go into the classification of surfaces with small invariants,
we want a convenient criterion to exclude this kind of exceptional surfaces in
the proof of Theorem \ref{theorem-main}:

\begin{proposition} \label{proposition-no-rational-curve}
Let $S$ be a normal projective Gorenstein surface with $p_g(S)=0$, and let $A$ be an ample Cartier divisor on $S$ such that
$$
A \simeq K_S + \Delta_S
$$
where $\Delta_S$ is an effective Weil divisor such that the pair $(S, \Delta_S)$ is lc. Assume that we have
$$
h^0(S, \sO_S(A)) \leq q(S).
$$
Then $\Delta_S$ does not contain any smooth rational curve. 
\end{proposition}

\begin{proof}
We argue by contradiction, so let $l \subset \Delta_S$ be a smooth rational curve.
In particular $\Delta_S \neq 0$ and therefore $A\cdot\Delta_S>0$.
By Proposition \ref{proposition-exceptional-cases-B} the surface has canonical singularities
and we have an Albanese fibration $\holom{\alpha}{S}{C}$ onto a smooth curve $C$ of
genus $g=q(S)$. Moreover we have
\begin{equation}
\label{help}
(r-1) \cdot (q(S)-1) \leq 1
\end{equation}
where $r= \rk (\alpha_* \sO_S(A))$.

{\em 1st case. Assume that $g>1$.} By Lemma \ref{lemma-pg-zero} the general $\alpha$-fibre $F$ is $\PP^1$. Since $A \cdot F \geq 1$ we have $r \geq 2$ and thus \eqref{help}
yields $g=2$ and $r=2$. Moreover $A \cdot F=1$ implies that $\holom{\alpha}{S}{C}$ is a $\PP^1$-bundle over the genus two curve $C$ \cite[II,Thm.2.8]{Kol96}. In particular the rational curve $l$ is a fibre 
of $\alpha$, thus a nef divisor on $S$. The pair $(S, \Delta_S-l)$ is lc, so
the direct image 
$$
\alpha_* \sO_S(K_{S/C}+\Delta_S-l) \simeq \alpha_* \sO_S(A) \otimes \sO_C(-K_C-l)
$$
is a nef vector bundle \cite[Thm.1.1]{Fuj17}. Thus
$\alpha_* \sO_S(A)$ has degree at least six. By the Riemann-Roch formula on the curve $C$
this implies $h^0(S, \sO_S(A))=h^0(C, \alpha_* \sO_S(A)) \geq 4$, a contradiction.

{\em 2nd case. Assume that $g=1$.} By Proposition \ref{proposition-exceptional-cases-B} we have
$h^0(S, \sO_S(A))=1$ and thus by Riemann-Roch 
$$
1 \geq \chi(S, \sO_S(A)) =\frac{1}{2} (K_S+\Delta_S) \cdot \Delta_S>0.
$$
Therefore $A \cdot \Delta_S = (K_S+\Delta_S) \cdot \Delta_S=2$ and Lemma \ref{lemma-configuration} applies: all the irreducible components of $\Delta_S$ are rational curves,
so they are contracted by the Albanese map. Moreover the support of $\Delta_S$ is connected, so it is contained in an $\alpha$-fibre $F_0$. Since $p_a(\Delta_S)=2$  we deduce that $p_a(F_0) \geq 2$. Since the arithmetic genus is constant in the flat family $\holom{\alpha}{S}{C}$
we obtain that $S$ is not uniruled. Let $\holom{\tau}{S'}{S}$ be the minimal resolution,
and let $\holom{\nu}{S'}{S_m}$ the minimal model of $S'$. Since $S$ has canonical singularities we have
$$
p_g(S_m)=p_g(S')=p_g(S)=0, \qquad q(S_m)=q(S')=q(S)>0.
$$
By \cite[Prop.VI.6]{Bea83book} the minimal surface $S_m$ does not contain any rational curves. Since the exceptional divisors of the resolution $\tau$ are rational curves the rigidity lemma \cite[Lemma 1.15]{Deb01} yields a birational map $\holom{\nu}{S}{S_m}$ that contracts all the rational curves on $S$. In particular it contracts the divisor $\Delta_S$. Yet $p_a(\Delta_S)=2>0$, a contradiction to the fact that the smooth surface $S_m$ has rational singularities.
\end{proof}

\begin{lemma} \label{lemma-configuration}
Let $S$ be a normal projective surface with canonical singularities and $p_g(S)=0$, and let $A$ be an ample Cartier divisor on $S$ such that
$$
A \simeq K_S + \Delta_S
$$
where $\Delta_S$ is an effective Weil divisor such that the pair $(S, \Delta_S)$ is lc
and $A \cdot \Delta_S=2$.
Assume that $\Delta_S$ has an irreducible component that is a smooth rational curve $l$.
Then the support of $\Delta_S$ is connected and all its irreducible components are rational curves.
\end{lemma}

\begin{proof}
Recall first that canonical surface singularities are Gorenstein, so both $K_S$ and $\Delta_S = A-K_S$ are Cartier. 
We have $p_a(\Delta_S)=2$, so the inclusion $l \subset \Delta_S$ must be strict.
Since $A$ is ample with $A \cdot \Delta_S=2$ we obtain
$$
\Delta_S = l + R
$$
with $R$ an irreducible curve such that $A \cdot R=1$. In particular we have by subadjunction
$$
\deg K_R \leq (K_S+R) \cdot R \leq (K_S+l+R) \cdot R = A \cdot R = 1,
$$
so $p_a(R)=\frac{1}{2} \deg K_R + 1  \leq 1$. In particular $\Delta_S$ is connected because otherwise 
$$
2= p_a(\Delta_S) = p_a(R) + p_a(l)  \leq 1
$$
yields a contradiction. 

Let $\holom{\tau}{S'}{S}$ be the minimal resolution and set 
$$
\Delta_{S'} := \tau^* \Delta_S
$$
Since $\Delta_S$ is connected, the cycle $\Delta_{S'}$ is connected.
Denote by $R' \subset \Delta_{S'}$ the strict transform of $R$. We claim that $R'$ is a rational curve, and this concludes the proof.

{\em Proof of the claim.}
Since $S$ has canonical singularities the minimal resolution $\tau$ is crepant, so we
have $K_{S'} \simeq \tau^* K_S$ and 
$$
K_{S'}+\Delta_{S'} = \tau^* (K_S+\Delta_S) = \tau^* A.
$$
Thus we have
$$
1 = \tau^* A \cdot R' = (K_{S'}+\Delta_{S'}) \cdot R' = (K_{S'}+R') \cdot R' + (\Delta_{S'}-R') \cdot R'.
$$
If we show that $(\Delta_{S'}-R') \cdot R'\geq 2$, then the adjunction formula yields the claim. Since $\Delta_{S'}$ is connected, we have $(\Delta_{S'}-R') \cdot R'\geq 1$, we will argue by contradiction to exclude the case where we have an equality. 

In this case $R'$ is a curve of arithmetic genus one that meets $\Delta_{S'}-R'$ transversally exactly in one point. In particular we have
$$
2 = p_a(\Delta_{S'}) = p_a(R') + p_a(\Delta_{S'}-R') = 1 + p_a(\Delta_{S'}-R').
$$
The curve $l \subset S$ being smooth it meets the exceptional divisor over every
point $p \in l \cap S_{sing}$ in exactly one point and the intersection is transverse. Thus
we have
$$
p_a(\Delta_{S'}-R') = p_a(l) + \sum_{p \in l \cap S_{sing}} \fibre{\tau}{p} = 0
$$
since the exceptional divisors of the minimal resolution of an ADE-singularity have arithmetic genus zero. This gives the final contradiction.
\end{proof}

\section{Positivity arguments} \label{section-positivity}

This section is the technical core of the paper. We will study the positivity properties of the Cartier classes $M_X$ and $B_X$ and successively improve our knowledge about the existence and singularities of effective divisors contained in these classes.

\begin{proposition} \label{proposition-B-not-nef}
In the situation of Setup \ref{setup}, the divisor $B_X$ is not nef.
If $B_X$ is effective, then $h^0(X, \sO_X(B_X))=1$ and $B_X$ is a normal prime divisor.
\end{proposition}

\begin{proof}
Consider the exact sequence
$$
0 \rightarrow \sO_X(-M_X) \rightarrow \sO_X(B_X) \rightarrow \sO_Y(B) \rightarrow 0.
$$
We have $h^0(X, \sO_X(-M_X))=0$ since otherwise $-M_X$, and thus its restriction $-M$ to the general anticanonical divisor $Y$, is effective. Thus we have an injection
$$
H^0(X, \sO_X(B_X)) \hookrightarrow H^0(Y, \sO_Y(B)).
$$
Since $h^0(Y, \sO_Y(B))=1$ by Fact \ref{fact-not-moving}, we have $h^0(X, \sO_X(B_X)) \leq 1$.
By  Lemma \ref{lemma-two-sections-fano} this implies that $B_X$ is not nef.

If $B_X$ is effective, it is a normal prime divisor since its restriction to the ample divisor $Y$ is the normal prime divisor $B$.
\end{proof}

\begin{proposition} \label{proposition-adjoint-nefness}
Let $Y$ be a $\Q$-factorial threefold with canonical Gorenstein singularities such that $c_1(Y)=0$. Let $A$ be an ample Cartier divisor with $h^0(Y, \sO_Y(A)) \geq 2$ such that $\Bs{ |A|}$ has pure dimension two. Let $A \simeq M+B$ be the decomposition into mobile part and fixed part $B$. Then $A+M$ is a nef and big divisor.
\end{proposition}

\begin{proof} 
We choose a general element $M$ in the mobile part, so 
by Theorem \ref{theorem-kawamata} the pair 
$(Y, M+B)$ is lc. If $A+M$ is not nef, there exists an irreducible curve $\gamma$ such that $(A+M) \cdot \gamma<0$. Since $A$ is ample, we have $M \cdot \gamma \leq -2$. 
In particular we have have $\gamma \subset \Bs{|M|}$. Since $|M|$ is the mobile part of
$|A|$ we have $\Bs{|M|} \subset \Bs{|A|}$. By assumption $\Bs{|A|}$ has pure dimension two,
so it coincides with $B$. Thus we have 
$\gamma \subset B,$
and hence 
$$
\gamma \subset M \cap B.
$$
The surface $M+B$ is slc by Theorem \ref{theorem-adjunction}, so it has normal crossing singularities in codimension one. Therefore $B$ and $M$ are smooth in the generic point of $\gamma$, in particular $\gamma$ is not contained in the singular locus of $M$. Since
$$
- K_M \cdot \gamma = - M \cdot \gamma \geq 2,
$$
we know by Fact \ref{fact-deformations-surface}
that the curve $\gamma$ deforms in $M$. In particular $M$ contains a positive-dimensional family of curves $\gamma_t$ such that $M \cdot \gamma_t<0$. Yet $M$ is a mobile divisor and $\dim Y=3$, so this is impossible.
\end{proof}

\begin{corollary} \label{corollary-adjoint-nefness}
In the situation of Setup \ref{setup},
the divisor $-K_X+M_X$ is nef and has
numerical dimension at least three. In particular
$$
H^q(X, \sO_X(M_X)) = 0 \qquad \forall \ q \geq 2.
$$
\end{corollary}

\begin{proof}
By Theorem \ref{theorem-kollar} we know that $-K_X+M_X$ is nef if and only if the restriction to $Y$ is nef.
Thus nefness follows from Proposition \ref{proposition-adjoint-nefness}. The numerical dimension is at least three since the restriction to $Y$ is big, so of numerical dimension three. Finally the vanishing follows from
$$
H^q(X, \sO_X(M_X)) = H^q(X, \sO_X(K_X + (-K_X+M_X)))
$$
and the numerical Kawamata-Viehweg vanishing theorem \cite[Example 4.3.7]{Laz04a}.
\end{proof}

\begin{proposition} \label{proposition-vanishing-adjoint}
In the situation of Setup \ref{setup}, assume that $M_X$ is not nef. Then we have
$$
H^q(X, \sO_X(K_X+M_X)) = 0 \qquad \forall \ q \geq 3.
$$
\end{proposition}

For the proof we start with a general lemma which is essentially contained in K{\"u}ronya's paper \cite{Kur13}:

\begin{lemma} \label{lemma-kuronya}
Let $Y$ be a normal projective threefold with canonical Gorenstein singularities, and let $D$
be an effective Cartier divisor on $Y$ such that for every irreducible component $D' \subset D$, the restriction $D|_{D'}$ is pseudoeffective. Then 
$$
H^q(Y, \sO_Y(K_Y+D)) = 0 \qquad \forall \ q \geq 2
$$
unless (maybe) $D$ is nef of numerical dimension at most one.
\end{lemma}

\begin{proof}
Since the restriction $D|_{D'}$ is pseudoeffective and $\dim D'=2$ there are at most finitely many curves $\gamma \subset Y$ such that $D \cdot \gamma<0$. In particular for a general very ample divisor $A \subset Y$, the restriction $D|_A$ is nef. Moreover
$$
(D|_A)^2 = D^2 \cdot A = D|_D \cdot A|_D \geq 0
$$
and equality holds if and only if the pseudoeffective class $D|_D$ is zero.
Hence if $(D|_A)^2 =0$, the restriction $D|_D$ is numerically trivial and $D$ is nef with numerical dimension at most one. 
If $(D|_A)^2 >0$ the restriction $D|_A$ is nef and big and we obtain the vanishing by \cite[Thm.C]{Kur13}\footnote{The statement in \cite{Kur13} is for a manifold, but it is straightforward to see that the proof works in our setup.}.
\end{proof}

\begin{proof}[Proof of Proposition \ref{proposition-vanishing-adjoint}]
We twist the exact sequence
$$
0 \rightarrow \sO_X(-Y) \simeq \sO_X(K_X) \rightarrow \sO_X \rightarrow \sO_Y \rightarrow 0
$$
with $M_X$ to get
$$
0 \rightarrow \sO_X(K_X+M_X) \rightarrow \sO_X(M_X) \rightarrow \sO_Y(M) \rightarrow 0.
$$
Since $M$ is mobile, it satisfies the conditions of Lemma \ref{lemma-kuronya}; by assumption $M_X$ (and thus $M$) is not nef, so we 
have $H^q(Y, \sO_Y(M))=0$ for $q \geq 2$. By the exact sequence in cohomology this implies
$$
H^q(X, \sO_X(K_X+M_X)) = H^q(X, \sO_X(M_X)) \qquad \forall \ q \geq 3.
$$
Yet the right hand side vanishes by Corollary \ref{corollary-adjoint-nefness}.
\end{proof}

\begin{corollary} \label{corollary-BX-effective-general}
In the situation of Setup \ref{setup}, assume that $M_X$ is not nef.
Then the divisor class $B_X$ is effective.
\end{corollary}

\begin{proof}
We twist the restriction sequence
$$
0 \rightarrow \sO_X(-Y) \simeq \sO_X(-(M_X+B_X)) \rightarrow \sO_X \rightarrow \sO_Y \rightarrow 0
$$
with $B_X$ to get
$$
0 \rightarrow \sO_X(-M_X) \rightarrow \sO_X(B_X) \rightarrow \sO_Y(B) \rightarrow 0.
$$
By Serre duality and Proposition \ref{proposition-vanishing-adjoint} one has
$$
H^1(X, \sO_X(-M_X)) = H^3(X, \sO_X(K_X+M_X)) = 0.
$$
Therefore $h^0(Y, \sO_Y(B))=1$ implies that $B_X$
is effective. 
\end{proof}

\begin{proposition} \label{proposition-MX-effective}
In the situation of Setup \ref{setup}, the divisor class $M_X$ is effective and mobile.
\end{proposition}

\begin{remark*}
If $M_X$ is nef, the linear system $|M_X|$ might have fixed components. Nevertheless the class being nef, it is mobile.
If $M_X$ is not nef the proof will show that 
$$
h^0(X, \sO_X(M_X)) = h^0(X, \sO_X(-K_X))-1. 
$$
\end{remark*}

\begin{proof}
If $M_X$ is nef, this is immediate from Lemma \ref{lemma-two-sections-fano}.
If $M_X$ is not nef we know by Corollary \ref{corollary-BX-effective-general} 
that $B_X$ is effective.
Thus $B_X$ is normal prime divisor by Proposition \ref{proposition-B-not-nef} and we can twist the restriction sequence for $B_X$ by $K_X+B_X$ to get an exact sequence
$$
0 \rightarrow \sO_X(K_X) \rightarrow \sO_X(K_X+B_X) \rightarrow \sO_{B_X}(K_X+B_X) \simeq
\sO_{B_X}(K_{B_X}) \rightarrow 0. 
$$
Taking cohomology we get a sequence
\begin{multline*}
\ldots \rightarrow H^3(X, \sO_X(K_X)) \rightarrow H^3(X, \sO_X(K_X+B_X)) \rightarrow H^3(B_X, 
\sO_{B_X}(K_{B_X}))  \\ 
\rightarrow H^4(X, \sO_X(K_X)) \rightarrow H^4(X, \sO_X(K_X+B_X)) 
\rightarrow 0
\end{multline*}
Using Serre duality on $X$ and $B_X$ this transforms into
\begin{multline*}
\ldots \rightarrow H^1(X, \sO_X)=0 \rightarrow H^3(X, \sO_X(K_X+B_X)) \rightarrow H^0(B_X, 
\sO_{B_X})=\C  \\ 
\rightarrow H^0(X, \sO_X)=\C \rightarrow H^0(X, \sO_X(-B_X))=0 
\rightarrow 0
\end{multline*}
In conclusion we get
$$
H^q(X, \sO_X(K_X+B_X))=0 \qquad \forall \ q \geq 3.
$$
Now we twist the restriction sequence to $Y$ by $M_X$ to get
$$
0 \rightarrow \sO_X(-B_X) \rightarrow \sO_X(M_X) \rightarrow \sO_Y(M) \rightarrow 0.
$$
By Serre duality and what precedes
$$
H^1(X, \sO_X(-B_X)) = H^3(X, \sO_X(K_X+B_X)) = 0.
$$
Thus the restriction map
$$
H^0(X, \sO_X(M_X)) \rightarrow H^0(Y, \sO_Y(M))
$$
is surjective. In particular 
$$
\Bs{|M|} = \Bs{|M_X|} \cap Y.
$$
Since $Y$ is ample and $M$ is mobile, this shows that $\dim \Bs{|M_X|} \leq 2$.
Thus $|M_X|$ is mobile. 
\end{proof}

The divisor $M_X$ being effective by Proposition \ref{proposition-MX-effective}, the divisor class $-K_X+M_X$ is big. Since $-K_X+M_X$ is also nef by Corollary \ref{corollary-adjoint-nefness} we finally obtain:

\begin{corollary} \label{corollary-adjoint-nef-big} 
In the situation of Setup \ref{setup}, the divisor $-K_X+M_X$ is nef and big.
\end{corollary}

\begin{proposition} \label{proposition-BX-lc}
In the situation of Setup \ref{setup}, assume that $h^0(X, \sO_X(B_X)) \neq 0$.
Then the pair $(X, B_X)$ is log-canonical. Moreover
the log-canonical centres are the prime divisor
$B_{X}$ and (possibly) some smooth curves $C \subset B_X$
such that $B_X \cdot C<0$.
\end{proposition}

\begin{proof}
We have $h^0(X, \sO_X(B_X)) = 1$ by Proposition \ref{proposition-B-not-nef}.
Thus $B_X$ is the unique effective divisor in its linear system, hence general. Since $(Y, B)
=(Y, Y \cap B_X)$ is log-canonical by Corollary \ref{corollary-restrictions-singularities}, we know by Lemma \ref{lemma-pair-lc}
that $(X, B_X)$ is lc.

Let us now describe its lc centres: since $B_X-(K_X+B_X)=-K_X$ is ample
we know by \cite[Thm.8.1]{Fuj11}
that for every lc centre $Z \subset X$ the restriction morphism
$$
H^0(X, \sO_X(B_X)) \rightarrow H^0(Z, \sO_Z(B_X)) 
$$
is surjective. Since $Z \subset B_X$ and $h^0(X, \sO_X(B_X))=1$ we obtain that $H^0(Z, \sO_Z(B_X))=0$.

This immediately excludes the possibility that $\dim Z=0$. 
Since $B_X$ is normal by Proposition \ref{proposition-B-not-nef}, there are no two-dimensional lc centres.

The unique remaining possibility is that $\dim Z=1$ and the centre is minimal. Thus 
by Kawamata subadjunction \cite{Kaw98} the curve $Z$ is normal and  there exist an effective divisor $\Delta_Z$ on $Z$ such that
$$
K_Z + \Delta_Z \sim_\Q (K_X+B_X)|_{Z}.
$$ 
In particular
$$
B_X|_Z - (K_Z + \Delta_Z) \sim_\Q -K_X|_Z
$$ 
is ample. Thus if $B_X \cdot Z \geq 0$, a Riemann-Roch computation shows that $h^0(Z, \sO_Z(B_X))>0$, a contradiction.
\end{proof}

We now come to the key technical result of this section:

\begin{lemma} \label{lemma-extension-BX}
In the situation of Setup \ref{setup}, assume that $h^0(X, \sO_X(B_X)) \neq 0$.
Then we have a surjective restriction morphism
$$
H^0(X, \sO_X(-K_X)) \rightarrow H^0(B_X, \sO_{B_X}(-K_X)).
$$
\end{lemma}

\begin{proof}
By Proposition \ref{proposition-BX-lc} the pair $(X,B_X)$ is lc, and $B_X$ is an lc centre for this pair. By Fujino's extension theorem \cite[Thm.1.11]{Fuj14} we have to show that
$$
-K_X - (K_X+B_X) = -K_X + M_X
$$
is nef and log big, i.e.\ the restriction of $-K_X+M_X$ to every lc centre of $(X, B_X)$ is big.

By Corollary \ref{corollary-adjoint-nef-big} the nef divisor $-K_X+M_X$ is big on $X$.
Since $M_X$ is mobile by Proposition \ref{proposition-MX-effective}, the restriction
$M_X|_{B_X}$ is pseudoeffective, so
$(-K_X+M_X)|_{B_X}$ is big.
By Proposition \ref{proposition-BX-lc} an lc centre $Z$ that is distinct from $B_X$ is a curve such that $B_X \cdot Z<0$. Therefore $M_X \cdot Z>0$
and hence $(-K_X+M_X)|_{Z}$ is ample.
\end{proof}

\begin{proposition} \label{proposition-injection}
In the situation of Setup \ref{setup}, assume that $h^0(X, \sO_X(B_X)) \neq 0$.
Then we have an injection
$$
H^0(B, \sO_B(-K_X)) \hookrightarrow H^1(B_X, \sO_{B_X}) \cong H^1(B, \sO_B).
$$
Moreover we have 
$$
h^0(B_X, \sO_{B_X}(-K_X))=1.
$$
\end{proposition}

\begin{proof}
We have $B=Y \cap B_X$, so $B \subset B_X$ is a Cartier divisor that is linearly equivalent to $-K_X|_{B_X}$. Thus we have an exact sequence
$$
0 \rightarrow \sO_{B_X} \rightarrow \sO_{B_X}(-K_X) \rightarrow
\sO_B(-K_X) \rightarrow 0.
$$
Let us first show that the restriction map
$$
H^0(B_X, \sO_{B_X}(-K_X)) \rightarrow H^0(B, \sO_B(-K_X))
$$
is zero. Since $B \subset \BsAX$ the restriction map 
$$
H^0(X, \sO_X(-K_X)) \rightarrow H^0(B, \sO_B(-K_X))
$$
is zero. Since $H^0(X, \sO_X(-K_X)) \rightarrow H^0(B_X, \sO_{B_X}(-K_X))$ is surjective by Lemma \ref{lemma-extension-BX}, we get the claim.

This already implies the second statement since the kernel of the restriction map is $H^0(B_X, \sO_{B_X})=\C$. Since the restriction map is zero, 
we have an injection 
$$
H^0(B, \sO_B(-K_X)) \rightarrow H^1(B_X, \sO_{B_X}).
$$
Now consider the exact sequence
$$
0 \rightarrow \sO_{B_X}(K_X) \rightarrow \sO_{B_X} \rightarrow
\sO_B \rightarrow 0.
$$
The pair $(X, B_X)$ is lc by Proposition \ref{proposition-BX-lc} and $B_X$ is normal by Proposition \ref{proposition-B-not-nef},
so $B_X$ is a threefold with lc singularities.
Since $K_X|_{B_X}$ is an antiample Cartier divisor, we can apply Kodaira vanishing \cite[Cor.6.6]{KSS10} to get
$$
H^1(B_X, \sO_{B_X}(K_X)) = 0 = H^2(B_X, \sO_{B_X}(K_X)).
$$
Thus we have an isomorphism $H^1(B_X, \sO_{B_X}) \simeq H^1(B, \sO_B)$
and the first statement follows.
\end{proof}

\begin{corollary} \label{corollary-BX-canonical}
In the situation of Setup \ref{setup}, assume that $h^0(X, \sO_X(B_X)) \neq 0$.
Then the pair $(X, B_X)$ is plt, i.e.\ the threefold $B_X$ has canonical Gorenstein singularities.
\end{corollary}

\begin{proof}
By Proposition \ref{proposition-injection} we have $h^0(B, \sO_B(-K_X)) \leq q(B)$.
By Corollary \ref{corollary-restrictions-singularities} the conditions
of Lemma \ref{proposition-exceptional-cases-B} are satisfied. Thus 
we know that $B$ has canonical singularities.
By inversion of adjunction (Theorem \ref{theorem-adjunction})
the pair $(B_X, B)=(B_X, Y \cap B_X)$
is plt near $B$. Thus $(B_X, 0)$ is plt (i.e.\ klt) near $B$. Since $B_X$ is Gorenstein this implies that $B_X$ has canonical singularities near $B$. Since $B = B_X \cap Y$ is an ample  divisor the non-canonical locus of $B_X$ is at most a finite set.

By Proposition \ref{proposition-BX-lc} the pair $(X, B_X)$ is lc and the lc centres of dimension at most two are irreducible curves $C$. Again by inversion of adjunction the non-canonical locus
of $B_X$ coincides with the union of lower-dimensional lc centres which has pure dimension one. By the first paragraph the non-canonical locus has dimension at most zero, so it is empty.
\end{proof}

We are now ready for the first reduction step in the proof of Theorem \ref{theorem-main}:

\begin{theorem} \label{theorem-nef}
In the situation of Setup \ref{setup}, the divisor $M_X$ is  nef.
\end{theorem}

\begin{proof}
We argue by contradiction and assume that $M_X$ is not nef. By the cone theorem there exists a $K_X$-negative extremal ray $\R^+ \gamma$ such that $M_X \cdot \gamma<0$. Since $M_X$ is mobile by Proposition \ref{proposition-MX-effective} the extremal ray is small. Thus by Kawamata's classification \cite[Thm.1.1]{Kaw89} a connected component of the exceptional locus
is a $\PP^2 \subset X$ such that $\sO_{\PP^2}(-K_X) \simeq \sO_{\PP^2}(1)$.
Thus the intersection $Y \cap \PP^2$ is either a line or the whole surface $\PP^2$.
In the latter case we would have $\PP^2 \subset M \subset Y$, in contradiction to the mobility
of $M$. Therefore $Y \cap \PP^2$ is a smooth rational curve $l$ such that $M \cdot l<0$.
Thus 
$$
l \subset \Bs{|M|} \subset \Bs{|-K_X|_Y|}=B
$$
shows that the effective divisor $M|_B$ contains a smooth rational curve.

Since $M_X$ is not nef, the divisor $B_X$ is effective by Corollary \ref{corollary-BX-effective-general}. Thus Proposition \ref{proposition-injection} implies that we have an injection
$$
H^0(B, \sO_B(-K_X)) \hookrightarrow H^1(B, \sO_B).
$$
By Corollary \ref{corollary-restrictions-singularities} the surface $B$ satisfies the conditions of Proposition \ref{proposition-no-rational-curve}. Thus the support of $M|_B$ does not contain a smooth rational curve, in contradiction to the first paragraph.
\end{proof}

\section{Auxiliary results about linear systems, part II}

Theorem \ref{theorem-nef} shows the nefness of the divisor class $M_X$ which by
Fact \ref{fact-bpf-application} will give us a morphism $X \rightarrow T$ that
adds additional structure to all the varieties appearing in our setup. In this section we prove further technical results that use these structures.

We start with a statement that is a variation of Koll\'ar's injectivity
theorem \cite[Thm.2.2]{Kol86}.

\begin{proposition} \label{proposition-extension-CY-curve}
Let $Y$ be a normal $\Q$-factorial projective variety with klt singularities 
such that $c_1(Y)=0$. Assume that $Y$ admits a fibration $\holom{\psi}{Y}{C}$
onto a smooth projective curve, and let $F$ be a general fibre. Let $A$ be a nef and big Cartier divisor such that 
$$
A \equiv m F + B
$$
with $B$ an effective divisor such that $(Y, B)$ is lc. If $m>1$, the restriction map
$$
H^0(Y, \sO_Y(A)) \rightarrow H^0(F, \sO_F(A))
$$
is surjective.
\end{proposition}

\begin{proof}
It is clearly sufficient to show that
$$
H^1(Y, \sO_Y(A - F)) = 0.
$$
Since $Y$ has klt singularities and $(Y, B)$ is lc, the pair $(Y, \epsilon B)$ is klt
for every $\epsilon < 1$ \cite[Cor.2.35(5)]{KM98}. In particular $(Y, \frac{1}{m} B)$ is klt.
Now we write
$$
A - F \equiv K_Y + (m-1) F + B =
(K_Y + \frac{1}{m} B) + (m-1) (F+\frac{1}{m}B) 
$$
Since $A \equiv m (F+\frac{1}{m}B)$ and $m>1$ the $\Q$-divisor class $(m-1) (F+\frac{1}{m}B)$
is nef and big. Now we conclude with Kawamata-Viehweg vanishing \cite[Thm.7.26]{Deb01}.
\end{proof}

\begin{lemma} \label{lemma-elliptic-injection}
Let $Y$ be a normal projective $\Q$-factorial threefold with terminal singularities such that $K_Y \simeq \sO_Y$. Suppose that $Y$ admits an elliptic fibration 
$$
\holom{\varphi}{Y}{T}
$$
onto a surface $T$. Let $A$ be a nef and big Cartier divisor such that we have $A \simeq M + B$ 
where $M \simeq \varphi^* M_T$ with $M_T$ a nef and big Cartier divisor on $T$
and $B$ is an effective divisor such that $B \subset \Bs{|A|}$.

Then we have an injection
\begin{equation} \label{injection-B}
H^0(B, \sO_B(K_B+M)) \hookrightarrow H^0(T, \sO_T(K_T+M_T)).
\end{equation}
\end{lemma}

\begin{proof}
Since $Y$ is a terminal, Gorenstein, $\Q$-factorial threefold, it is factorial \cite[Lemma 5.1]{Kaw88}. Thus all the Weil divisors on $Y$ are Cartier.

Consider the exact sequence
$$
0 \rightarrow \sO_Y(M) \rightarrow \sO_Y(A) \rightarrow \sO_B(A) \rightarrow 0
$$
and the long exact sequence in cohomology
$$
H^0(Y, \sO_Y(A)) \rightarrow H^0(B, \sO_B(A)) \rightarrow H^1(Y, \sO_Y(M)) \rightarrow H^1(Y, \sO_Y(A))
$$
By Kawamata-Viehweg vanishing we have $H^1(Y, \sO_Y(A))=0$. Since $B \subset \Bs{|A|}$ the restriction map $H^0(Y, \sO_Y(A)) \rightarrow H^0(B, \sO_B(A))$ is zero, so we have an isomorphism
\begin{equation} \label{hzerotoh1ter}
H^0(B, \sO_B(K_B+M)) \simeq H^0(B, \sO_B(A)) \simeq H^1(Y, \sO_Y(M)).
\end{equation}
By the Leray spectral sequence we have an exact sequence
$$
0 \rightarrow H^1(T, \varphi_* \sO_Y(M)) \rightarrow H^1(Y, \sO_Y(M))
\rightarrow H^0(T, R^1 \varphi_* \sO_Y(M)) 
$$
By the canonical bundle formula \cite{Amb05} there exists a boundary divisor $\Delta_T$
on $T$ such that $(T, \Delta_T)$ is klt and $K_T+\Delta_T \sim_\Q 0$. Since $\varphi_* \sO_Y(M)
\simeq \sO_T(M_T)$ is nef and big we can apply 
the Kawamata-Viehweg vanishing theorem \cite[Thm.7.26]{Deb01} to see that
$$
H^1(T, \varphi_* \sO_Y(M)) = 0.
$$
By the projection formula  we have
$$
R^1 \varphi_* \sO_Y(M) \simeq R^1 \varphi_* \sO_Y \otimes \sO_T(M_T),
$$
so $R^1 \varphi_* \sO_Y(M)$ is torsion-free by \cite[Thm.2.1]{Kol86}.
 
Since $K_Y \simeq \sO_Y$ and $\varphi$ is flat over a big open susbset of $T$ we
have by relative duality \cite{Kle80}
$$
(R^1 \varphi_* \sO_Y) \simeq (\varphi_* \sO_Y(K_{Y/T}))^* \simeq \sO_T(K_T)
$$
in the complement of finitely many points. 
Since $R^1 \varphi_* \sO_Y(M)$ is torsion-free
the morphism 
$$
R^1 \varphi_* \sO_Y(M) \simeq \sO_T(K_T) \otimes \sO_T(M_T) \rightarrow
\sO_T(K_T+M_T) 
$$
is thus injective and yields an inclusion
$$
H^0(T, R^1 \varphi_* \sO_Y(M))  \rightarrow H^0(T, \sO_T(K_T+M_T)). 
$$
The statement now follows from \eqref{hzerotoh1ter}.
\end{proof}

We need the following characterisation of Example \ref{example-q-one}, a singular variant of \cite[Thm.0.2]{Fuj84} and its proof.

\begin{proposition} \label{proposition-fujita}
Let $S$ be a normal projective surface with canonical singularities
such that $H^1(S, \sO_S) \neq 0$. 
Let $A$ be an ample, effective Cartier divisor such that $A^2=1$ and $K_S \cdot A=-1$.
Then $S$ is a $\PP^1$-bundle $f: \PP(V) \rightarrow C$ over an elliptic curve $C$
such that $A$ is the tautological divisor.
\end{proposition}

\begin{proof}
Since $K_S \cdot A=-1$ the canonical bundle is not pseudoeffective, so $S$ is uniruled.
Since $H^1(S, \sO_S) \neq 0$ the surface is not rationally connected, and 
the MRC fibration coincides with the Albanese map $f: S \rightarrow \mbox{Alb}(S)$.
In particular $f$ contracts all the rational curves on $S$ and $\pi_1(S) \simeq \pi_1(C)$
where $C$ is the image of $f$. Our goal is to show that $q(S)=1$ and $f$ is a $\PP^1$-bundle.

Since the Cartier divisor $A$ is effective and $A^2=1$, the divisor $A$ is an integral curve.
By the adjunction formula the arithmetic genus of $A$ is one, so either $A$ is smooth
or its normalisation is $\PP^1$. 
Since the ample divisor $A$ is not contracted by the fibration $f$, we see that $A$
is a smooth elliptic curve. Now recall that canonical surface singularities
are hypersurface singularities, so by Goresky-MacPherson's Lefschetz theorem for homotopy groups \cite[Thm.3.1.21,Rem.3.1.41]{Laz04a}, the morphism 
$$
\Z^2 \simeq \pi_1(A) \rightarrow \pi_1(S)
$$
is surjective, in particular $q(S)=1$ and $C$ is also an elliptic curve. Since 
$$
\pi_1(A) \rightarrow \pi_1(S) \simeq \pi_1(C)
$$
is surjective, the \'etale map map $A \rightarrow C$ is an isomorphism. Thus $A$
is a section of $f$. Since $A$ is ample and has degree one on the fibres, all the $f$-fibres are integral curves. Thus $f$ is a $\PP^1$-bundle by \cite[II,Thm.2.8]{Kol96}.  Since $A \cdot f=1$, we have
$S \simeq \PP(f_* \sO_S(A))$.
\end{proof}

\begin{remark} \label{remark-fujita} 
An elementary computation \cite[V, Prop.2.21]{Har77} shows that in the situation of Proposition \ref{proposition-fujita}, the divisor $A$ is adjoint, i.e we have
$$
A \simeq K_S + \Delta_S
$$
with $\Delta_S$ an ample divisor with $\Delta_S^2=3$.
\end{remark}

\begin{proposition} \label{proposition-identify-B}
Let $Z$ be a normal projective threefold with canonical Gorenstein singularities such
that $q(Z)=1$ and $-K_Z$ is nef. Let $S \subset Z$ be a normal surface with canonical singularities such that $S$ is an ample Cartier divisor in $Z$.
Suppose that
$$
h^0(Z, \sO_Z(S))=1, \qquad \mbox{and} \qquad S \cdot (-K_Z)^2 \geq 2.
$$
Then either $K_S \equiv 0$ or $S$ is a ruled surface over an elliptic curve.
\end{proposition}

\begin{proof}
Let $\holom{\tau}{Z'}{Z}$ be a terminalisation of $Z$. Then $-K_{Z'} \simeq \tau^* (-K_Z)$
is nef and $h^i(Z', \sO_{Z'})=h^i(Z, \sO_Z)$ for all $i \in \N$.
Since $\tau^* S \cdot (-K_{Z'})^2 = S \cdot (-K_Z)^2>0$ the nef divisor $-K_{Z'}$ is not trivial. Thus $Z'$ is uniruled and hence $h^3(Z', \sO_{Z'})=0$. Since $q(Z)=1$ by assumption we obtain $\chi(Z', \sO_{Z'}) \geq 0$. Since $\tau^* S$ is nef and big and $-K_{Z'}$ is nef we have
$$
H^j(Z', \sO_{Z'}(\tau^* S)) = H^j(Z', \sO_{Z'}(K_{Z'} + (-K_{Z'}+\tau^* S))) = 0 \qquad \forall \ j \geq 1
$$
by Kawamata-Viehweg vanishing. 
Thus the Riemann-Roch formula (Fact \ref{fact-RR}) yields
\begin{multline}
1 = h^0(Z', \sO_{Z'}(\tau^* S)) = \chi(Z', \sO_{Z'}(\tau^* S))
\\
\geq \frac{1}{12} \tau^* S \cdot \tau^* (S-K_Z) \cdot \tau^* (2S-K_Z) + \frac{1}{12} \tau^* S \cdot c_2(Z')
\end{multline}
Since $-K_{Z'}$ is nef and $\tau^* S$ is nef we have $\tau^* S \cdot c_2(Z') \geq 0$
by \cite[Cor.1.5]{Ou23}. Applying the projection formula we obtain
$$
S \cdot  (S-K_Z) \cdot (2S-K_Z) \leq 12.
$$
Set $A:=S|_S$ and $H=-K_Z|_S$. Then $A$ is an ample Cartier divisor on $S$ and $H$
is a nef and big Cartier divisor on $S$ such that $H^2 \geq 2$. 
By the Hodge index inequality $(A \cdot H)^2 \geq A^2 \cdot H^2$ \cite[Prop.2.5.1]{BS95} this implies $A \cdot H \geq 2$.
Since
$$
12 \geq S \cdot  (S-K_Z) \cdot (2S-K_Z) = 2 A^2 + 3 A \cdot H + H^2 \geq 4 + 3 A \cdot H
$$
we actually have $A \cdot H=2$. Moreover we have $A^2 \leq 2$.

{\em 1st case. Suppose that $A^2=2$.}
Then we have equality in the Hodge index inequality and therefore 
$A \equiv H$ by \cite[Cor.2.5.4]{BS95}. Since 
$$
K_S \simeq (K_Z+S)|_S \simeq -H + A
$$
we obtain $K_S \equiv 0$.

{\em 2nd case. Suppose that $A^2=1$.}
Since
$$
K_S \simeq (K_Z+S)|_S \simeq -H + A
$$
we have $K_S \cdot A=-1$. Finally we have the Riemann-Roch inequality
$$
h^0(S, \sO_S(A)) \geq \frac{1}{2} A^2 + \frac{1}{2} (-K_S \cdot A) + \chi(\sO_S).
$$
Since $S$ is an ample divisor in a threefold and $Z$ has canonical singularities we can use Kodaira vanishing to shows that $q(S)=q(Z)=1$.
In particular $\chi(\sO_S) \geq 0$ and by the Riemann-Roch inequality
$h^0(S, \sO_S(A))>0$, so $A$ is an effective divisor.
Thus the polarised surface $(S, A)$ satisfies the conditions
of Proposition \ref{proposition-fujita}.
\end{proof}

\section{The nef case} \label{section-nef-case}

The goal of this section is to show

\begin{theorem} \label{theorem-not-nef-case}
In the situation of setup \ref{setup}, the divisor $M_X$ is not nef.
\end{theorem}

\begin{setup} \label{setup-nef}
{\rm For the proof of Theorem \ref{theorem-not-nef-case} we will argue by contradiction and assume that $M_X$ is nef. By the Fact \ref{fact-bpf-application} there exists a morphism with connected fibres
\begin{equation}
\label{the-fibration}
\holom{\varphi}{X}{T}
\end{equation}
such that $M_ X \simeq \varphi^* M_T$ for some ample Cartier divisor $M_T \rightarrow T$. 
Note that $B_X$ is $\varphi$-ample since $B_X \sim_\varphi -K_X$.

Since $X$ is a Fano there exists a boundary divisor such that $(X, \Delta)$ is klt and
$\Delta \sim_\Q -K_X$. Thus we apply Ambro's theorem \cite{Amb05} to see that $T$ is klt.
Note that it is not clear whether $T$ is $\Q$-factorial. }
\end{setup}

\begin{lemma} \label{lemma-vanish-MX1}
In the situation of Setup \ref{setup-nef} we have $H^1(X, \sO_X(-M_X))=0$ unless $T \simeq \PP^1$ and $M_T \simeq \sO_{\PP^1}(m)$ with $m \geq 2$.
\end{lemma}

\begin{proof}
Since $-K_X$ is $\varphi$-ample we have $R^j \varphi_* \sO_X=0$ for every $j \geq 1$ by relative Kodaira vanishing. Since $M_X \simeq \varphi^* M_T$ we deduce that $R^j \varphi_* \sO_X(-M_X)=0$ for $j \geq 1$
by the projection formula. Thus we have
$$
H^1(X, \sO_X(-M_X)) \simeq H^1(T, \sO_T(-M_T))
$$
and the latter is zero by Kodaira vanishing on the klt space $T$ unless $T$ is a curve. The rest is now straightforward.
\end{proof}

\begin{corollary} \label{corollary-BX-effective}
In the situation of Setup \ref{setup-nef}
we have $H^0(X, \sO_X(B_X)) \neq 0$ unless $T \simeq \PP^1$ and $M_T \simeq \sO_{\PP^1}(m)$ with $m \geq 2$.
\end{corollary}

\begin{proof}
Immediate from Lemma \ref{lemma-vanish-MX1} and the exact sequence
$$
0 \rightarrow \sO_X(-M_X) \rightarrow \sO_X(B_X) \rightarrow \sO_Y(B) \rightarrow 0.
$$
\end{proof}

We will prove Theorem \ref{theorem-not-nef-case} by making a case distinction in terms of the dimension of the base $T$.

\begin{proposition} \label{proposition-dim-one}
In the situation of Setup \ref{setup-nef} we have $\dim T > 1$.
\end{proposition}

\begin{proof}
If $\dim T=1$ the fibration $\varphi$ has as general fibre a smooth Fano threefold $F$ (the divisor $-K_F \simeq -K_X|_F$ is ample) and $T \simeq \PP^1$.

Choose $Y \in \AX$ a general element, and denote by $\holom{\psi}{Y}{\PP^1}$ the restriction of $\varphi$ to $Y$. Denote by $F_Y$ the general $\psi$-fibre, that is $F_Y= F \cap Y$.
We have a commutative diagram
$$
\xymatrix{
H^0(X, \sO_X(-K_X)) \ar[d] \ar @{->>}[r] & H^0(Y, \sO_Y(-K_X)) \ar[d]^{r_Y}   
\\
H^0(F, \sO_F(-K_X)) \ar @{->>}[r] & H^0(F_Y, \sO_{F_Y}(-K_X))  
}
$$
and the horizontal maps are surjective since $q(X)=0=q(F)$.

{\em 1st case. Suppose that $\sO_T(M_T) \simeq \sO_{\PP^1}(m)$ with $m \geq 2$.}
We have $-K_X|_Y \simeq \psi^* M_T + B$, moreover the pair $(Y,B)$ is lc
by Corollary \ref{corollary-restrictions-singularities}.
Thus by Proposition \ref{proposition-extension-CY-curve} the restriction map
$r_Y$ is surjective. Since $-K_X|_Y \simeq \psi^* M_T + B$
and $B$ is a fixed component of the linear system $|-K_X|_Y|$,
the image of $r_Y$ is generated by a section vanishing on $B \cap F_Y$.
Thus $r_Y$ is a surjective map of rank one and therefore $h^0(F_Y, \sO_{F_Y}(-K_X))=1$. The restriction map $H^0(F, \sO_F(-K_X)) \rightarrow H^0(F_Y, \sO_{F_Y}(-K_X))$ 
has kernel $H^0(F, \sO_F) \simeq \C$, so we deduce $h^0(F, \sO_F(-K_X))=2$. Yet $-K_F \simeq -K_X|_F$ and for a smooth Fano threefold we always have $h^0(F, \sO_F(-K_F)) \geq 3$
by \cite[Cor.2.1.14]{IP99}.

{\em 2nd case. Suppose that $\sO_T(M_T) \simeq \sO_{\PP^1}(1)$.}
By Corollary \ref{corollary-BX-effective} this implies that $B_X$ is an effective divisor.
Since $-K_X \simeq \varphi^* M_T + B_X\simeq F + B_X$, the divisor $B_X$ is relatively ample.
We claim that 
$$
H^1(Y, \sO_Y(-K_X|_Y-F_Y)) = H^1(Y, \sO_Y(B)) = 0.
$$
As in the first case this yields the surjectivity of $H^0(Y, \sO_Y(-K_X)) \rightarrow H^0(F_Y, \sO_{F_Y}(-K_X))$ and the desired contradiction.

{\em Proof of the claim.}
By Proposition \ref{proposition-B-not-nef} we know that $B_X$ is not nef,
so there exists a $K_X$-negative extremal ray $\R^+ \gamma$ such that
$B_X \cdot \gamma<0$. Since $B_X$ is effective, the corresponding contraction 
$$
\holom{\mu}{X}{X'}
$$
must be birational with exceptional locus contained in $B_X$. Since $B_X$ is $\varphi$-ample
and $\mu$-antiample,
the intersection of any $\mu$-fibre with a $\varphi$-fibre must be finite. Thus the fibres of $\mu$ have dimension at most one. 
By Ando's theorem \cite[Thm.2.3]{And85} this implies that $\mu$ is a smooth blowup along a surface. In particular we have $(-K_X+B_X) \cdot \gamma=0$, so $-K_X+B_X$ is non-negative on the extremal ray $\R^+ \gamma$. Since $\gamma$ was an arbitrary $B_X$-negative extremal ray the cone theorem
implies that $-K_X+B_X$ is nef. Since $B_X$ is effective, the divisor $-K_X+B_X$ is nef and big. Therefore
$$
H^j(X, \sO_X(B_X)) = H^j(X, K_X + \sO_X(-K_X+B_X)) = 0 
$$ 
for $j \geq 1$ by Kawamata-Viehweg vanishing. Now consider the exact sequence
$$
0 \rightarrow \sO_X(-M_X) \rightarrow \sO_X(B_X) \rightarrow \sO_Y(B) \rightarrow 0.
$$
Since $M_X \simeq \varphi^* M_T$ and $\dim T=1$ we have 
$H^2(X, \sO_X(-M_X)) \simeq H^2(T, \sO_T(-M_T)) = 0.$ By the long exact sequence in cohomology the map 
$0=H^1(X, \sO_X(B_X)) \rightarrow H^1(Y,\sO_Y(B))$ is surjective and 
we are finally done.
\end{proof}

\begin{proposition} \label{proposition-dim-high}
In the situation of Setup \ref{setup-nef} we have $\dim T < 3$.
\end{proposition}

\begin{proof}
Assume that $\dim T \geq 3$. Since $Y$ is an ample divisor, the morphism
$\varphi|_Y: Y \rightarrow T$ is generically finite onto its image, so $M \simeq \varphi^* M_T|_Y$ is nef and big. 
Consider the exact sequence
$$
0 \rightarrow \sO_Y(M) \rightarrow \sO_Y(-K_X) \rightarrow \sO_B(-K_X) \rightarrow 0.
$$
Since $H^1(Y, \sO_Y(M))=0$ by Kawamata-Viehweg vanishing the restriction morphism
$$
H^0(Y, \sO_Y(-K_X)) \rightarrow H^0(B, \sO_B(-K_X))
$$
is surjective. Since $B$ is in the base locus of $|-K_X|_Y|$ we obtain $H^0(B, \sO_B(-K_X))=0$. 

Since $\dim T \geq 3$ we have $h^0(X, \sO_X(B_X)) \neq 0$ by Corollary \ref{corollary-BX-effective}. Thus Proposition \ref{proposition-injection} shows that we have an injection
$$
H^0(B, \sO_B(-K_X)) \hookrightarrow H^1(B, \sO_B).
$$
By Corollary \ref{corollary-restrictions-singularities} the surface $B$ satisfies
the conditions of Proposition \ref{proposition-exceptional-cases-B}. Thus we have $q(B)>0$
and the inclusion $H^0(B, \sO_B(-K_X)) \hookrightarrow H^1(B, \sO_B)$ is an equality.
Since the first space has dimension zero, this is a contradiction.
\end{proof}

\begin{proposition} \label{proposition-nef-dim-two}
In the situation of Setup \ref{setup-nef} we have $\dim T \neq 2$
\end{proposition}

This is the part of the proof that requires the most work. We will use properties
of the anticanonical divisor $Y$ and the effective divisor $B_X$ to determine some of the invariants of the base $T$, then we will use the smoothness of $X$ to reach a contradiction.

\begin{proof}
Assume that $\dim T=2$. By Corollary \ref{corollary-BX-effective} the divisor $B_X$ is effective. By Corollary \ref{corollary-BX-canonical} the divisor $B_X$ has canonical Gorenstein singularities, and by Proposition \ref{proposition-injection} we have
$$
h^0(B_X, \sO_{B_X}(-K_X))=1 \quad \text{and} \quad q(B_X)=q(B).
$$ 
Since $Y \subset X$ is an ample divisor, the fibration $\varphi$ induces an elliptic fibration
$$
\holom{\psi}{Y}{T}
$$
such that $-K_X|_Y \simeq \psi^* M_T + B$. 
By Theorem \ref{theorem-known} and Corollary \ref{corollary-factorial} the threefold $Y$
satisfies the conditions of Lemma \ref{lemma-elliptic-injection}. Hence we have an injection
\begin{equation}
\label{the-injection}
H^0(B, \sO_B(K_B+M)) \hookrightarrow H^0(T, \sO_T(K_T+M_T)).
\end{equation}
Observe that the induced morphism $\varphi|_{B_X}$ is surjective onto $T$
since $B_X$ is $\varphi$-ample.

{\em Step 1. We show that $h^0(T, \sO_T(K_T+M_T))=0$ unless $K_T \simeq -M_T$.}
By the adjunction formula 
$$
-K_{B_X} \simeq M_X|_{B_X} \simeq (\varphi^* M_T)|_{B_X}
$$
is nef with numerical dimension two.
Since $B_X$ has canonical singularities we can apply \cite[Thm.3.1]{FG12} to see that there exists a boundary divisor $\Delta_T$ on $T$ such that
$(T, \Delta_T)$ is klt and 
$$
K_{B_X} \sim_\Q (\varphi^* (K_T+\Delta_T))|_{B_X}.
$$
Thus we have $K_T+\Delta_T \equiv -M_T$ and hence $K_T+M_T=-\Delta_T$ is not pseudoeffective
unless $\Delta_T=0$.
Therefore $h^0(T, \sO_T(K_T+M_T)) =0$ unless $\Delta_T=0$. 
In the latter case $h^0(T, \sO_T(K_T+M_T)) \neq 0$
implies that $K_T \simeq -M_T$, in particular $K_T$ is Cartier.

{\em Step 2. We reach a contradiction.}
By Proposition \ref{proposition-injection} we have an injection
$$
H^0(B, \sO_B(-K_X)) \hookrightarrow H^1(B, \sO_B).
$$
By Corollary \ref{corollary-restrictions-singularities} the surface $B$ satisfies the conditions of Proposition \ref{proposition-exceptional-cases-B}. In particular $B$ has canonical singularities, positive irregularity $q(B)>0$, and the inclusion above is an equality. Since 
$$
H^0(B, \sO_B(-K_X)) = H^0(B, \sO_B(K_B+M)) 
$$
the injection \eqref{the-injection} and the first step show that 
$$
H^0(B, \sO_B(-K_X)) = H^1(B, \sO_B) = H^0(T, \sO_T(K_T+M_T) = \C.
$$
In particular $K_T \simeq -M_T$ by the first step, i.e $T$ is a del Pezzo surface with at most canonical singularities.

Since $B$ is the complete intersection of $B_X$ and $Y$,
the surface $B \subset B_X$ is an ample Cartier divisor and the restriction
of $\varphi|_{B_X}$ to $B$ is still surjective onto the surface $T$. Since $q(T)=0$ and $q(B)=1$, and the irregularity is a birational invariant of varieties with rational singularities, the map 
$$
\tau:= \varphi|_T: B \rightarrow T
$$
is not birational. Thus $\tau$ is {\em generically} finite of degree at least two and 
$$
(M_X|_B)^2 = (\tau^* M_T)^2 = \deg \tau \cdot M_T^2 \geq 2.
$$
Thus $q(B_X)=q(B)=1$, $-K_{B_X} \simeq M_X|_{B_X}$ is nef, $B\subset B_X$ is an ample Cartier divisor with $h^0(B_X, \sO_{B_X}(B))=h^0(B_X, \sO_{B_X}(-K_X))=1$ and $B\cdot(-K_{B_X})^2\geq 2$. Thus the threefold $B_X$ and the surface $B \subset B_X$ satisfy the conditions of Proposition \ref{proposition-identify-B},
and we have two cases:

{\em Case a) We have $K_B \equiv 0$.} Since $K_B \simeq B|_B$ this implies that $B$ is nef.
Thus $B_X$ is nef by Theorem \ref{theorem-kollar}, a contradiction to Proposition \ref{proposition-B-not-nef}.

{\em Case b) $B$ is a ruled surface over an elliptic curve.}
Since $\rho(\PP(V))=2$ and the map $\tau$ is surjective, the 
del Pezzo surface $T$ has Picard number at most two (canonical singularities are $\Q$-factorial, \cite[Prop.4.11]{KM98}). 
By Remark \ref{remark-fujita} the nef and big class $M_X|_B$ is ample and $(M_X|_B)^2=3$. Since $M_X|_B \simeq \tau^* M_T \simeq \tau^* (-K_T)$ is ample,
the generically finite map $\tau$ is actually finite.
Moreover, $\tau$ having degree at least two, we deduce that $\tau$ has degree three
and $(-K_T)^2=1$. Thus $T$ is a del Pezzo surface of degree one and Picard number at most two,
in particular it is not smooth. By \cite[Prop.1.3]{AW95} this implies that
the fibration $\holom{\varphi}{X}{T}$ is not equidimensional, so there exists
a prime divisor $D \subset X$ that is contracted onto a point in $T$. 

Since the effective divisors $B_X$ and $Y$ are $\varphi$-ample, the intersection
$$
D \cap B_X \cap Y = D \cap B
$$
is non-empty. Hence there exists a curve $E \subset B$ that
is contracted by $\tau$. Yet we showed above that $\tau$
is finite, a contradiction.
\end{proof}

\begin{remark*}
It may seem annoying that the case of a fibration $X \rightarrow T$ with $T$ a del Pezzo surface of degree one requires so much additional effort. Note however that this situation is very close to Example \ref{example1}, so our arguments must be specific enough to rule out this situation.
\end{remark*}

\section{The conclusion.}

\begin{proof}[Proof of Theorem \ref{theorem-main}]
Assume that a general anticanonical divisor $Y \in \AX$ is 
$\Q$-factorial, so we satisfy the Assumption \ref{main-assumption} from the Setup \ref{setup}. By Theorem \ref{theorem-nef} the divisor $M_X$ is nef, yet this contradicts Theorem \ref{theorem-not-nef-case}.
\end{proof}

\vspace{.5cm}
\textbf{Competing interests:} the authors declare none.

\providecommand{\bysame}{\leavevmode\hbox to3em{\hrulefill}\thinspace}
\providecommand{\MR}{\relax\ifhmode\unskip\space\fi MR }
\providecommand{\MRhref}[2]{
  \href{http://www.ams.org/mathscinet-getitem?mr=#1}{#2}
}
\providecommand{\href}[2]{#2}

\end{document}